\documentclass[11pt,english]{article}
\def\margin_comment#1{\marginpar{\sffamily{\tiny #1\par}\normalfont}}
\usepackage{graphicx}
\usepackage{setspace}
\date{}
\usepackage{amsmath,amssymb,amsthm}
\usepackage{babel}
 \newtheorem{thm}{Theorem}[section]
 \numberwithin{equation}{section} 
 \numberwithin{figure}{section} 
 \theoremstyle{plain}
  \newtheorem*{thm*}{Theorem}
 \theoremstyle{definition}
\theoremstyle{plain}

 \newtheorem*{defn*}{Definition}
 \theoremstyle{plain}
 \newtheorem{prop}[thm]{Proposition} 
 \theoremstyle{remark}
 \newtheorem{ex}[thm]{Example}
 \theoremstyle{remark}
 \newtheorem*{ex*}{Example}
 \theoremstyle{remark}
 
\theoremstyle{plain}
 \newtheorem{claim}[thm]{Claim}
\theoremstyle{plain}
 
 \theoremstyle{plain}
 \newtheorem{lem}[thm]{Lemma} 
 \theoremstyle{definition}
 \newtheorem{defn}[thm]{Definition}
\newtheorem*{acknowledgment}{Acknowledgment}
{\catcode`\|=\active
  \gdef\Pres#1{\left\langle\:{\mathcode`\|"8000\let|\SetVert
#1}\:\right\rangle}}
\def\SetVert{\egroup\;\middle|\;\bgroup}

\AtBeginDocument{
  
}
\begin{document}
\title{The conjugacy problem in semigroups and monoids}
\author{Fabienne Chouraqui}
\maketitle
\begin{abstract}
We present an algorithmic approach to the conjugacy problems in  monoids and semigroups, using rewriting systems. There is a  class of monoids and semigroups that satisfy the condition that
the transposition problem and the left and right conjugacy problem are equivalent.
 The free monoid and the completely simple semigroups belong to this class. We give a solution to the conjugacy problem for monoids and semigroups
  in this class that are presented by a complete rewriting system that satisfies some additional conditions.
   \end{abstract}

\section{Introduction}
The use of string rewriting systems or Thue systems has been proved to be a very efficient tool to solve the word problem. Indeed, Book shows that there is a linear-time algorithm to decide the word problem for a monoid that is defined by a finite and complete rewriting system \cite{book_linear}.
A question that arises naturally is whether the use of rewriting systems may be an efficient tool for solving other decision problems, specifically  the conjugacy problem. Several authors have studied this question, see \cite{naren_otto2,naren_otto}, \cite{otto}, and \cite{pedersen}. The complexity of this question is due to some  facts, one point is that for monoids the conjugacy problem and the word problem are independent one of another \cite{otto}, this is different from the situation for groups.
Another point  is that in  semigroups and monoids,  there are several different notions of conjugacy that are not equivalent in general. We describe them in the following.\\
Let $M$ be a monoid (or a semigroup)  generated by $\Sigma$ and let $u$ and $v$ be two words in  the free monoid $\Sigma^{*}$. The right conjugacy problem asks  if there  is  a word  $x$ in the free monoid $\Sigma^{*}$ such that $xv=_{M}ux$, and  is  denoted by $\operatorname{RConj}$. The left conjugacy problem asks if there is a word  $y$  in the free monoid $\Sigma^{*}$ such that  $vy=_{M}yu$, and is denoted by $\operatorname{LConj}$. The conjunction of the left and the right conjugacy problems is denoted by $\operatorname{Conj}$. The relations $\operatorname{LConj}$ and $\operatorname{RConj}$ are reflexive and transitive but not necessarily symmetric, while $\operatorname{Conj}$ is  an equivalence relation.
A different generalization of conjugacy asks if there are   words  $x,y$  in the free monoid such that  $u=_{M}xy$ and $v=_{M}yx$. This is called the \emph{transposition problem} and it is denoted by  $\operatorname{Trans}$.
This relation is reflexive and symmetric, but not necessarily transitive.\\
 In general, if the answer to this question is positive then the answer
to the above questions is also positive, that is $\operatorname{Trans} \subseteq \operatorname{Conj} \subseteq \operatorname{LConj},\operatorname{RConj}$.  For free monoids, Lentin and Schutzenberger show that $\operatorname{Trans}= \operatorname{Conj}= \operatorname{LConj}=\operatorname{RConj}$ \cite{schutz} and  for monoids with a special presentation (that is all the relations have the form $r=1$) Zhang shows that $\operatorname{Trans}= \operatorname{RConj}$  \cite{zhang}. We denote by $\operatorname{Trans^{*}}$ the transitive closure of $\operatorname{Trans}$. Choffrut shows that  $\operatorname{Trans^{*}}= \operatorname{Conj}= \operatorname{LConj}=\operatorname{RConj}$ holds in a free inverse monoid $FIM(X)$ when restricted to the set of non-idempotents \cite{choffrut}. He shows that $\operatorname{LConj}$ is an equivalence relation on $FIM(X)$ and he proves the decidability of this problem in this case. Silva generalized the results of Choffrut to a certain class of one-relator inverse monoids. He proves the decidability of $\operatorname{Trans}$ for $FIM(X)$ with one idempotent relator \cite{silva}.\\
In this work, we use rewriting systems in order to solve  the conjugacy problems presented above in some semigroups and monoids.
A special rewriting system satisfies the condition that all the rules have the form $l \rightarrow 1$, where $l$ is any word. Otto  shows that $\operatorname{Trans}= \operatorname{Conj}= \operatorname{LConj}$
for a monoid with a special complete rewriting system and that $\operatorname{Trans}$ is an equivalence relation. Moreover,  he shows that whenever the rewriting system is finite then the conjugacy problems are solvable \cite{otto}. Narendran and Otto show that $\operatorname{LConj}$ and $\operatorname{Conj}$ are decidable for a finite, length-decreasing and complete rewriting system \cite{naren_otto} and that $\operatorname{Trans}$
is not decidable \cite{naren_otto2}. We describe our approach to solve the conjugacy problems using rewriting systems in the following.\\
    Let $M$ be the finitely presented monoid $\operatorname{Mon}\langle\Sigma\mid R\rangle$ and let $\Re$
be a complete and reduced rewriting system for $M$.
Let $u$ be a word in  $\Sigma^{*}$, we consider $u$ and all its cyclic conjugates in  $\Sigma^{*}$,   $\{u_{1}=u, u_{2},..,u_{k}\}$,  and we apply on each element $u_{i}$ rules from $\Re$ (whenever this is possible).
We say that a word $u$ is \emph{cyclically irreducible} if $u$ and all its cyclic conjugates are irreducible modulo $\Re$.
If for  some $1 \leq i \leq n$, $u_{i}$  reduces to $v$, then we say that $u$ \emph{cyclically reduces to $v$} and we denote it by $u \looparrowright v$, where $\looparrowright$ denotes a binary relation on the words in $\Sigma^{*}$. A question that arises naturally is when $u$ and all its cyclic conjugates cyclically reduce to the same cyclically irreducible element (up to cyclic conjugation in $\Sigma^{*}$), denoted by $\rho(u)$.\\
 We  define on $\looparrowright$ the properties of terminating and confluent in a very similar  way as for  $\rightarrow$ and we find that if $\looparrowright$ is terminating and confluent then each word reduces to a unique cyclically irreducible element. Moreover, we give    a partial solution to the conjugacy problems presented above in the following way:  if  $u$ and $v$ are transposed, then
   $\rho(u)$  and $\rho(v)$ are cyclic conjugates in $\Sigma^{*}$
 and this implies in turn that  $u$ and $v$ are left and right conjugates.
 So, in semigroups and monoids in which $\operatorname{Trans}= \operatorname{Conj}$, there is a solution to the conjugacy problems.  A \emph{completely simple semigroup} is a semigroup that has no non-trivial two-sided ideals and that possesses minimal one-sided ideals. Using the results of McKnight and Storey in \cite{macknight}, it holds that $\operatorname{Trans}= \operatorname{Conj}$ in a completely simple semigroup.
 So, in the case of completely simple semigroups and monoids with a  finite special complete rewriting system, our result gives a solution to the conjugacy problems,  whenever $\looparrowright$ is terminating and confluent.\\
 The paper is organized as follows. In Section $2$, we define the  binary relation $\looparrowright$ on the words in $\Sigma^{*}$ such that,  for $u,v$ in $\Sigma^{*}$,  $u\looparrowright v$ if $u$ cyclically reduces to $v$. We  define on $\looparrowright$ the properties of terminating and confluent in a very similar way as for  $\rightarrow$ and we show that if $\looparrowright$ is terminating and confluent then each word reduces to  a  unique cyclically irreducible element. In Section $3$, we establish the connection between a terminating and confluent relation $\looparrowright$ and the conjugacy problems. We show that,  for $u,v$ in $\Sigma^{*}$,  if $u$ and $v$ are transposed then they have the same cyclically irreducible form (up to conjugacy in  $\Sigma^{*}$) and this implies in turn that $u$  and $v$
are left and right conjugates. \\
In Section $3$, we adopt a kind of local approach as it is very difficult to decide wether a relation $\looparrowright$ is terminating, we define there the notion of triple that is $\widetilde{c}$-defined. In Section $5$, we give a necessary condition for the confluence of $\looparrowright$, given that it terminates. In Section $6$, using the results from Section $5$, we give an algorithm of cyclical completion that is very much inspired by the Knuth-Bendix algorithm of completion. Given a terminating relation  $\looparrowright$, if it is not confluent then some new cyclical reductions are added in order to obtain an equivalent relation $\looparrowright^{+}$ that is terminating and confluent. At last, in Section $7$, we address the case of length-preserving rewriting systems. All along this paper, $\Re$ denotes a complete and reduced rewriting system, not necessarily a finite one.

\begin{acknowledgment}
This work is a part of  the author's PhD research, done  at the Technion under
the supervision of Professor Arye Juhasz.  I am very grateful to
Professor Arye Juhasz, for his patience, his encouragement and his many helpful remarks.
I am also grateful to Professor Stuart Margolis for his comments on this result.
\end{acknowledgment}
\section{Definition of the relation $\looparrowright$}

Let $\Sigma$ be a non-empty set. We denote by $\Sigma^{*}$ the free monoid
generated by $\Sigma$; elements of $\Sigma^{*}$ are finite sequences
called \emph{words} and the empty word will be denoted by 1.
A \emph{rewriting system} $\Re$ on $\Sigma$ is a set of ordered
pairs in $\Sigma^{*}\times \Sigma^{*}$.
If $(l,r) \in \Re$ then for any words $u$ and $v$ in $\Sigma^{*}$,
 we say that the word $ulv$ \emph{reduces} to the word $urv$ and we
write $ulv\rightarrow urv$ . A word $w$ is said to be \emph{reducible}
if there is a word $z$ such that $w\rightarrow z$. If there is no
such $z$ we call $w$ \emph{irreducible}. A rewriting system $\Re$ is called \emph{terminating} \emph{(or
Noetherian)} if there is no infinite sequence of reductions
$w_{1}\rightarrow w_{2}\rightarrow...\rightarrow w_{n}\rightarrow...$.
We denote by  ``$\rightarrow^{*}$''  the reflexive
transitive closure of the relation ``$\rightarrow$''.
A rewriting system $\Re$ is called \emph{confluent} if for any words $u,v,w$ in $\Sigma^{*}$ ,
$w\rightarrow^{*}u$ and $w\rightarrow^{*}v$ implies that there is
a word $z$ in $\Sigma^{*}$ such that $u\rightarrow^{*}z$ and
$v\rightarrow^{*}z$ (that is if $u$ and $v$ have a common ancestor
then they have a common descendant). A rewriting system $\Re$ is called \emph{complete (or convergent)} if $\Re$ is
terminating and confluent. If $\Re$ is complete then every
word $w$ in $\Sigma^{*}$ has a unique irreducible equivalent word
that is called the \emph{normal form} of $w$.
We say that $\Re$ is \emph{reduced} if for any rule
$l\rightarrow r$ in $\Re$ , $r$ is irreducible and there is no
rule $l'\rightarrow r'$ in $\Re$ such that $l'$ is a subword of
$l$. If $\Re$ is complete then there exists a reduced and complete
rewriting system $\Re$'  which is equivalent to $\Re$ \cite{squier}.
We refer the reader to \cite{book,sims} for more details.\\
Let $\operatorname{Mon}\langle \Sigma \mid R \rangle$ be a finitely presented monoid $M$ and let $\Re$
be a complete rewriting system for $M$. Let $u$ and $v$ be elements in $\Sigma^{*}$. We define the following binary relation
$ u \circlearrowleft^{1}v$ if  $v$ is a cyclic conjugate of $u$ obtained by moving the first letter of $u$ to be the last
letter of $v$. We define $u\circlearrowleft^{i} v$ if $v$ is a cyclic conjugate of $u$ obtained from $i$
successive applications of $\circlearrowleft^{1}$.
We allow $i$ being $0$ and in this case  if $u\circlearrowleft^{0} v$ then $v=u$
 in the free monoid $\Sigma^{*}$. As an example, let $u$ be the word $abcdef$ in
$\Sigma^{*}$. If $u\circlearrowleft^{1} v$ and
$u\circlearrowleft^{4} w$, then $v$
is the word $bcdefa$ and $w$ is the word $efabcd$ in $\Sigma^{*}$.\\

We now translate the operation of taking cyclic conjugates and reducing them using the rewriting system $\Re$ in terms of a binary relation. We  say that \emph{$u$ cyclically reduces to $v$} and we write
\begin{gather}
u\looparrowright v
\end{gather} if there is a sequence
\begin{gather}
u\circlearrowleft^{i} \widetilde{u}\rightarrow v
\end{gather}
From its definition, the relation $\looparrowright$ is \emph{not} compatible with concatenation. We define by $\looparrowright^{*}$ the reflexive and transitive closure of $\looparrowright$, that is $u\looparrowright^{*} v$ if there is a sequence $u\looparrowright u_{1}\looparrowright
u_{2}\looparrowright...u_{k-1}\looparrowright v$. We call such a sequence \emph{a sequence of cyclical reductions}. We say that a sequence of cyclical reductions is \emph{trivial} if it has the following form:
$u\circlearrowleft^{i} u_{1}\rightarrow^{*}u_{1}\circlearrowleft^{j} u_{2}\rightarrow^{*}u_{2}\circlearrowleft^{k}..$. Of course, we are interested only in non-trivial sequences. We use the following notation:\\
- $\widetilde{u}$ denotes a cyclic conjugate of $u$ in the free monoid $\Sigma^{*}$.\\
- $u\bumpeq v$ if $u$ and $v$ are cyclic conjugates in the free monoid $\Sigma^{*}$.\\
- $u =_{M} v$ if the words $u$ and $v$ are equal as elements in $M$.\\
- $u = v$ if the words $u$ and $v$ are equal in the free monoid $\Sigma^{*}$.

   Now, we define the properties of terminating and confluent for $\looparrowright$ in a very similar way as it is done for $\rightarrow$, and we define the cyclically irreducible form of a word.
\begin{defn}
We say that $\Re$ is \emph{cyclically terminating} or that the
relation $\looparrowright$ is \emph{terminating} if there is no (non-trivial)
infinite sequence of cyclical reductions, that is  there is no infinite sequence
$u_{1}\looparrowright u_{2}\looparrowright...u_{n}\looparrowright...$
\end{defn}
\begin{ex}\label{ex_not_termin_no_cyc_irred}
Let $\Re=\{ ab\rightarrow bc,cd\rightarrow da\}$,  $\Re$ is a complete and finite  rewriting system.
Let consider the word $bcd$, then we have  $bcd\rightarrow
bda\circlearrowleft^{2} abd\rightarrow bcd \rightarrow..$, that is there is an infinite sequence of cyclical reductions. So, $\Re$ is not
cyclically terminating.
\end{ex}

\begin{defn}
We say that a word $u$ is \emph{cyclically irreducible} if $u$ and all its cyclic conjugates are irreducible modulo $\Re$, that is  there is no $v$ in $\Sigma^{*}$ such that
$u\looparrowright v$ (unless $u \bumpeq v$). We define a \emph{cyclically irreducible
form of $u$} (if it  exists) to be a
cyclically irreducible word $v$ (up to $\bumpeq$) such that
$u\looparrowright^{*}v$.
We denote by $\rho(u)$ a cyclically irreducible
form of $u$, if it exists.
\end{defn}
\begin{ex}
Let $\Re=\{ab\rightarrow bc,cd\rightarrow da\}$ as before. From Ex. \ref{ex_not_termin_no_cyc_irred},  $bcd$ does not have any cyclically irreducible form.  But, the word $acd$ has a unique cyclically irreducible form $ada$ since $acd\rightarrow ada$ and no rule from $\Re$ can be applied on $ada$ or on  any cyclic conjugate of $ada$ in $\Sigma^{*}$.
\end{ex}
\begin{lem}
\label{lem:term_canon} If $\Re$ is cyclically terminating, then
each word in $\Sigma^{*}$ has at least one cyclically irreducible form.
\end{lem}
\begin{proof}
Let $u$ be a word in $\Sigma^{*}$. Since $\Re$ is cyclically
terminating any sequence of cyclical reductions terminate, that is
$u\looparrowright u_{1}\looparrowright u_{2}...\looparrowright
u_{k}$. So, $u_{k}$ is a cyclically irreducible form of $u$.
\end{proof}
\begin{defn}
 We say that $\Re$ is \emph{cyclically confluent}
 or that the relation $\looparrowright$ is
\emph{confluent} if for any words $u,v,w$ in $\Sigma^{*}$,
$w\looparrowright^{*}u$ and $w\looparrowright^{*}v$ implies that
there exist cyclically conjugates  words $z$  and $z'$ in $\Sigma^{*}$ such that
$u\looparrowright^{*}z$ and $v\looparrowright^{*}z'$.
We say that $\Re$ is \emph{locally cyclically confluent}
 or that the relation $\looparrowright$ is
\emph{locally confluent} if for any words $u,v,w$ in $\Sigma^{*}$,
$w\looparrowright u$ and $w\looparrowright v$ implies that
there exist cyclically conjugates  words $z$  and $z'$ in $\Sigma^{*}$ such that
$u\looparrowright^{*}z$ and $v\looparrowright^{*}z'$, where $z\bumpeq z'$.
\end{defn}
We have the following equivalence between the local confluence of $\looparrowright$ and the  confluence of $\looparrowright$, given that $\looparrowright$ is terminating. We omit the proof as it is completely similar to the proof of the equivalence between the local confluence of $\rightarrow$ and the  confluence of $\rightarrow$, given that $\rightarrow$ is terminating and we refer the reader to \cite{book}.
\begin{claim}\label{claim_local_conf}
Let $\Re$ be a complete and reduced rewriting system and assume that $\Re$ is cyclically terminating. Then $\Re$ is cyclically confluent if and only if $\Re$ is locally cyclically confluent.
\end{claim}

\begin{ex}\label{ex_not_cyc_confluent}
In \cite{hermiller+meier},  Hermiller and Meier construct a finite and complete rewriting system for the group  $\operatorname{Gp} \langle a,b \mid aba=bab \rangle$, using another set of generators. For the monoid with the same presentation,  the  set of generators is: $\{a,b,\underline{ab},\underline{ba},\Delta=\underline{aba}\}$ and the complete and finite rewriting system is $\Re=\{ ab \rightarrow \underline{ab},ba\rightarrow\underline{ba},
a\underline{ba}\rightarrow \Delta, \underline{ab}a\rightarrow \Delta, b\underline{ab}\rightarrow \Delta, \underline{ab}\,\underline{ab}\rightarrow a \Delta, \underline{ba}b\rightarrow \Delta,
\underline{ba}\,\underline{ba}\rightarrow b\Delta, \Delta a\rightarrow b\Delta, \Delta b\rightarrow a\Delta, \Delta \underline{ab}\rightarrow \underline{ba}\Delta,\Delta\underline{ba}\rightarrow \underline{ab}\Delta
\}$. Let consider  the word $ab$, then  $ab\rightarrow \underline{ab}$ and $ab\circlearrowleft^{1}ba\rightarrow \underline{ba}$. That is, $ab\looparrowright \underline{ab}$ and $ab\looparrowright \underline{ba}$, where both $\underline{ab}$ and $\underline{ba}$ are cyclically irreducible, so $\Re$ is not cyclically confluent (nor locally cyclically confluent).
\end{ex}
\begin{lem}
\label{lem:conf_canon}
For any word in $\Sigma^{*}$, the cyclical confluence of $\Re$
ensures the existence of at most one cyclically irreducible form (up to $\bumpeq$).
\end{lem}
\begin{proof}
Let $w$ be a word in $\Sigma^{*}$ such that  $w\looparrowright^{*}u$ and $w\looparrowright^{*}v$,
then from the cyclically confluence of $\Re$, we have that there
is a word $z$ in $\Sigma^{*}$ such that $u$ and $v$ cyclically
reduce to $z$ and $z'$ respectively, where $z \bumpeq z'$. So, $w$ has at most one cyclically irreducible form (up to $\bumpeq$).
\end{proof}
\begin{defn}
 $\Re$ is called \emph{cyclically complete} if $\Re$ is cyclically terminating
and cyclically confluent.
\end{defn}
Lemmas \ref{lem:term_canon} and \ref{lem:conf_canon} have the following direct consequence.
\begin{prop}
\label{claim_cyc_complete_uniq_irred}
If $\Re$ is cyclically complete, then any word $w$ in $\Sigma^{*}$ has a unique cyclically irreducible form. Moreover, if $\widetilde{w} \bumpeq w$, then $w$ and $\widetilde{w}$ have the same cyclically irreducible
form (up to $\bumpeq$).
\end{prop}

\section{The relation $\looparrowright$ and the conjugacy problems}
Let $M$ denote the finitely  presented monoid  $\operatorname{Mon} \langle \Sigma\mid R \rangle$ and assume $M$ has a complete and reduced rewriting system $\Re$. Whenever $\Re$ is cyclically complete, we  solve  partially the transposition problem and  the left and right conjugacy problems  in the following sense: we show that, given two words $u$ and $v$ in $\Sigma^{*}$, if $u$ and $v$ are transposed then they have the same cyclically irreducible form (up to $\bumpeq$) and this implies in turn  that $u$ and $v$ are left and right conjugates. We give examples that show that the converses are not necessarily true.
Note that given words $u$ and $v$ if we write $u\looparrowright v$ or $u\looparrowright^{*} v$, we assume
implicitly that this is done in a finite number of steps.
 We denote by
$u\equiv_{M} v$ the following equivalence relation: there are words $x,y$ in $\Sigma^{*}$ such that $ux=_{M} xv$ and $yu=_{M}vy$, that is  $u$ and $v$ are left and right conjugates.
We describe in the following lemmas the connection between the relation  $\looparrowright$ and the conjugacy problems.
 \begin{lem}\label{lem_uloopvimplies_transp}
 Let $u$ and $v$ be words in $\Sigma^{*}$ such that $u\looparrowright^{*}
v$ and $u \neq_{M} v$. Assume that the sequence of cyclical reductions has the following form:
$u\circlearrowleft^{i} \widetilde{u} \rightarrow^{*}v$. Then $u$ and $v$ are transposed.
\end{lem}
\begin{proof}
 If $u$ and $\widetilde{u}$ are the same, then $u=_{M}v$. Otherwise,
  $u\bumpeq \widetilde{u}$ (not trivially), that is there are words $x,y$ in
 $\Sigma^{*}$ such that $u=xy$ and $ \widetilde{u}=yx$ in
 $\Sigma^{*}$. Then $v=_{M} \widetilde{u}=yx$, that is $u$ and $v$ are transposed.
 \end{proof}

\begin{lem}\label{lem_uloopvimplies conj}
 Let $u$ and $v$ be words in $\Sigma^{*}$ such that $u\looparrowright^{*}
v$. Then $u \equiv_{M} v$.
\end{lem}
\begin{proof}
Since $u\looparrowright^{*}
v$, there is a sequence of cyclical reductions $u=u_{1}\circlearrowleft^{i} \widetilde{u} \rightarrow^{*}u_{2}\circlearrowleft^{i} \widetilde{u_{2}} \rightarrow^{*}u_{3}... \rightarrow^{*}u_{k}=v$. From lemma \ref{lem_uloopvimplies_transp},  $u_{1}$ and $u_{2}$ are transposed, $u_{2}$ and $u_{3}$ are transposed, .., $u_{k-1}$ and $u_{k}$ are transposed. So, this implies that $u_{i} \equiv_{M} u_{i+1}$ for $1 \leq i \leq k-1$ and since $\equiv_{M}$ is transitive we have that $u \equiv_{M} v$.
\end{proof}

\begin{prop}\label{prop_samerho_conj}
Let $M$ denote the  finitely  presented monoid  $\operatorname{Mon} \langle \Sigma\mid R \rangle$ and  let $\Re$ be a complete and reduced rewriting system for $M$.
Let $u$ and $v$ be words in $\Sigma^{*}$ and assume  that they cyclically reduce to a same
cyclically irreducible form (up to $\bumpeq$), that is $\rho(u) \bumpeq \rho(v)$. Then $u \equiv_{M} v$
\end{prop}
\begin{proof}
From lemma \ref{lem_uloopvimplies conj}, $u \equiv_{M} \rho(u)$ and $v \equiv_{M} \rho(v)$. Since $\rho(u) \bumpeq \rho(v)$ and  $\equiv_{M}$ is an
equivalence relation,  $u \equiv_{M} v$.
\end{proof}
The converse is not true in general, namely  $u \equiv_{M} v$ does not imply that $\rho(u) \bumpeq \rho(v)$. Indeed, let consider the following example. Let  $\Re=\{ bab\rightarrow aba,ba^{n}ba\rightarrow aba^{2}b^{n-1}, n \geq 2\}$. Then $\Re$ is a complete and infinite  rewriting system for the braid monoid presented by
$\operatorname{Mon}\langle a,b \mid aba=bab \rangle$.
It holds that $a\equiv_{M}b$, since $a(aba)=_{M}(aba)b$ and
$(aba)a=_{M}b(aba)$, but $\rho(a)=a$ and $\rho(b)=b$ and they are not cyclic conjugates.
This example is due to Patrick Dehornoy.

\begin{lem}\label{lem_equal_same_cycl}
Let $\Re$ be a complete, reduced and cyclically complete rewriting system for $M$. Let $u$ and $v$ be words in $\Sigma^{*}$. If $u=_{M}v$, then $\rho(u)\bumpeq \rho(v)$.
\end{lem}
\begin{proof}
Assume by contradiction that $u$ and $v$ do not have the same
cyclically irreducible form, that is  $u \looparrowright^{*} z$
and $v \looparrowright^{*} z'$, where $z,z'$ are cyclically irreducible
and not cyclic conjugates in $\Sigma^{*}$.
Since $\Re$ is a complete  rewriting system and $u=_{M}v$, there is an irreducible word $w$ such
 that $u \rightarrow ^{*} w$ and $v \rightarrow ^{*} w$. \\
  We have the following diagram:
 $\begin{array}{cccccccc}
 u & \looparrowright^{*}& z\\
 &\searrow^{*} \\
 && w \\
 & \nearrow^{*} \\
v & \looparrowright^{*}& z'
  \end{array}$\\
  Assume with no loss of generality that $w \looparrowright^{*} z$, so $v \looparrowright^{*} z$.
  But $v \looparrowright^{*} z'$ and $\Re$ is cyclically complete, so a contradiction.
   Note that if $w \looparrowright^{*} z''$, where  $z'' \neq z,z'$,  then $u \looparrowright^{*} z$ and $u \looparrowright^{*} z''$, also a contradiction.
  \end{proof}
\begin{thm}
Let $\Re$ be a complete, reduced and cyclically
complete rewriting system for $M$. Let $u$ and $v$ be words in
$\Sigma^{*}$. \\
$(i)$ If $u$ and $v$ are transposed, then $\rho(u) \bumpeq \rho(v)$.\\
$(ii)$ If $\rho(u) \bumpeq \rho(v)$, then $u \equiv_{M} v$.
\end{thm}

\begin{proof}
$(i)$ Since $u$ and $v$ are transposed, there are words $x$ and $y$ in
$\Sigma^{*}$ such that $u =_{M}xy$ and $v=_{M}yx$.
Since $xy \bumpeq yx$ and $\Re$ is cyclically complete, $\rho(xy) \bumpeq \rho(yx)$, from Proposition \ref{claim_cyc_complete_uniq_irred}.
From lemma \ref{lem_equal_same_cycl}, $\rho(xy) \bumpeq \rho(u)$ and the same holds for $v$ and $yx$.
So, $\rho(u) \bumpeq \rho(v)$.
$(ii)$ holds from Proposition \ref{prop_samerho_conj} in a more general context.
\end{proof}

\section{A local approach for $\looparrowright$: definition of $\operatorname{Allseq}$$(w)$}
Given a complete and reduced rewriting system $\Re$, it is a very hard task to determine if $\Re$ is cyclically terminating, since we have to check a potentially infinite number of words. So, we adopt a kind of local approach, that is for each word $w$  in $\Sigma^{*}$ we consider all the possible sequences of cyclical reductions that begin by each word from $\{w_{1},..,w_{k}\}$, where  $w_{1}=w,w_{2},..,w_{k}$ are all the cyclic conjugates of $w$ in $\Sigma^{*}$. We call the set of all these sequences \emph{$\operatorname{Allseq}$$(w)$}. We say that $\operatorname{Allseq}$$(w)$ \emph{terminates} if there is no infinite sequence of cyclical reductions in $\operatorname{Allseq}$$(w)$.
Clearly, $\Re$ is cyclically terminating if and only if $\operatorname{Allseq}$$(w)$ terminates for every $w$ in $\Sigma^{*}$. We illustrate this idea with an example.
\begin{ex}\label{ex_braid_B3}
Let $\Re=\{ bab\rightarrow aba,ba^{n}ba\rightarrow aba^{2}b^{n-1},$ where $n\geq 2\}$.
Then $\Re$ is a complete and infinite  rewriting system for the braid monoid presented by
$\operatorname{Mon}\langle a,b \mid aba=bab \rangle$.
We denote by $w$ the word $ba^{2}ba$. We have the following infinite
sequence of cyclical reductions: $ba^{2}ba \rightarrow aba^{2}b
\circlearrowleft^{1} ba^{2}ba$, that is  $\operatorname{Allseq}$$(w)$ does not terminate.
This holds also for $ba^{n}ba$ for each $n \geq 2$.
\end{ex}
 We say that  $\operatorname{Allseq}$$(w)$ \emph{converges} if a unique cyclically irreducible form is achieved in $\operatorname{Allseq}$$(w)$ (up to $\bumpeq$). Clearly, if $\Re$ is cyclically confluent then $\operatorname{Allseq}$$(w)$ converges for every $w$ in $\Sigma^{*}$. The converse is true only if $\Re$ is cyclically terminating. We illustrate this with an example.
\begin{ex}\label{ex_braid_uniquecyclicalform}
Let $\Re=\{ bab\rightarrow aba,ba^{n}ba\rightarrow aba^{2}b^{n-1},$ where $n\geq 2\}$ as in Ex. \ref{ex_braid_B3}.
 It holds that  $\operatorname{Allseq}$$(ba^{2}ba)$ does not terminate (see Ex. \ref{ex_braid_B3}). Yet, $\operatorname{Allseq}$$(ba^{2}ba)$ converges, since  $a^{3}ba$ is the unique  cyclically irreducible form achieved in $\operatorname{Allseq}$$(w)$. Indeed,  there is the following sequence of cyclical reductions:  $ba^{2}ba \circlearrowleft^{1}a^{2}bab \rightarrow a^{3}ba$ and   all the cyclic conjugates of $w$ cyclically reduce to $a^{3}ba$. So,  although $\operatorname{Allseq(ba^{2}ba)}$ does not terminate,  a unique cyclically irreducible form $a^{3}ba$ is achieved.
\end{ex}
We find a condition that ensures that $\operatorname{Allseq}$$(w)$ converges, given that $\operatorname{Allseq}$$(w)$ terminates. Before we proceed, we give the following definition.
\begin{defn}
Let $\Re$ be a complete, reduced  rewriting system and let $w$ be a word in $\Sigma^{*}$. Let $w_{1},w_{2}$ be cyclic conjugates
of $w$ in $\Sigma^{*}$ and let $r_{1}$ and $r_{2}$ be rules in $\Re$ such that $r_{1}$ can be applied on $w_{1}$ and $r_{2}$ can be applied on $w_{2}$.
We say that \emph{the triple $(w,r_{1},r_{2})$ is $\widetilde{c}$-defined} if there is a cyclic conjugate $\widetilde{w}$ of $w$ such that both rules $r_{1}$ and $r_{2}$ can be applied on
$\widetilde{w}$. We allow an empty entry in a triple $(w,r_{1},r_{2})$, that is $r_{1}$ or $r_{2}$ is empty.
\end{defn}
\begin{ex}\label{ex_def_triple}
Let $\operatorname{Mon}\langle x,y,z \mid xy=yz=zx \rangle$, this is the Wirtinger presentation of the trefoil knot group. Let $\Re= \{xy \rightarrow zx, yz \rightarrow zx, xz^{n}x  \rightarrow zxzy^{n-1}, n \geq 1\}$ be a complete and infinite rewriting system for this monoid (see \cite{chou1}). Let consider the word $yxz^{2}x$, $yxz^{2}x$ and $xyxz^{2}$ are cyclic conjugates on which the rules $xz^{2}x \rightarrow zxzy$ and $xy \rightarrow zx$ can be applied respectively. We claim that the triple $(yxz^{2}x,xz^{2}x \rightarrow zxzy, xy \rightarrow zx)$ is  $\widetilde{c}$-defined.
Indeed, there is the cyclic conjugate $xz^{2}xy$ on which both the rules $xz^{2}x \rightarrow zxzy$ and $xy \rightarrow zx$ can be applied. But, as an example the triple  $(xz^{2}xz^{3},xz^{2}x \rightarrow zxzy, xz^{3}x \rightarrow zxzy^{2})$ is not $\widetilde{c}$-defined.
\end{ex}
In what follows, we show that if  $\operatorname{Allseq}$$(w)$ terminates and all the triples occurring there are  $\widetilde{c}$-defined, then  $\operatorname{Allseq}$$(w)$  converges. The following lemma is the induction basis of the proof. For brevity, we write $u \looparrowright^{r_{1}} v_{1}$ for $u\circlearrowleft u_{1}  \rightarrow^{r_{1}} v_{1}$, where $u_{1}  \rightarrow^{r_{1}} v_{1}$ means that $v_{1}$ is obtained from the application of the rule $r_{1}$ on $u_{1}$.
\begin{lem}\label{lem_one_step}
Let $(w,r_{1},r_{2})$ be a triple and assume that $(w,r_{1},r_{2})$ is $\widetilde{c}$-defined.
 Assume that  $w \looparrowright^{r_{1}} v_{1}$ and  $w \looparrowright^{r_{2}} v_{2}$, then there are cyclically conjugates words $z_{1}$ and $z_{2}$ such that $v_{1} \looparrowright^{*} z_{1}$ and
$v_{2} \looparrowright^{*} z_{2}$.
\end{lem}
\begin{proof}
 We denote by $l_{1}$ and $l_{2}$ the left-hand sides of the rules $r_{1}$ and $r_{2}$ respectively and  by $m _{1}$ and $m_{2}$ the corresponding right-hand sides. Then $l_{1}$ has an occurrence in $w_{1}$ and $l_{2}$ has an occurrence in $w_{2}$, where $w_{1} \bumpeq w_{2} \bumpeq w$ and we use the symbol $*$ to denote an occurrence in a word.
 Since $(w,r_{1},r_{2})$ is $\widetilde{c}$-defined,  there exists  $\widetilde{w}$ such that  $\widetilde{w} \bumpeq w$ and $l_{1}$ and $l_{2}$  both have an occurrence in $\widetilde{w}$. Then one of the following holds:\\
 $(i)$ $\widetilde{w}= x*l_{1}*y*l_{2}*s$, where $x,y,s$ are words.\\
  $(ii)$ $\widetilde{w}= x*l_{2}*y*l_{1}*s$, where $x,y,s$ are words.\\
 $(iii)$ $\widetilde{w}= x*l_{1}*l''_{2}*y$, where $x,y$ are words, $l_{1}=l'_{1}l''_{1}$, $l_{2}=l'_{2}l''_{2}$ and $l''_{1}=l'_{2}$.\\
 $(iv)$ $\widetilde{w}= x*l_{2}*l''_{1}*y$, where $x,y$ are words, $l_{1}=l'_{1}l''_{1}$, $l_{2}=l'_{2}l''_{2}$ and $l''_{2}=l'_{1}$.\\
The word $l_{1}$ cannot be a subword of $l_{2}$ (or the converse),  since $\Re$ is reduced and there is no inclusion ambiguity.
We check the cases $(i)$ and $(iii)$  and the other two cases are symmetric.
If both  $l_{1}$ and $l_{2}$ have an  occurrence in $w_{1}$ and in $w_{2}$, then obviously  there are words $z_{1}$ and $z_{2}$ such that $v_{1} \looparrowright^{*} z_{1}$ and
$v_{2} \looparrowright^{*} z_{2}$, where  $z_{1} \bumpeq z_{2}$.  So, assume that
$l_{1}$ has no occurrence in $w_{2}$ and $l_{2}$ has no occurrence in $w_{1}$.\\
       Case $(i)$: Assume that $\widetilde{w}= x*l_{1}*y*l_{2}*s$. Then the words $w_{1}$ and $w_{2}$ have the following form:
$w_{1}= l''_{2}* sx*l_{1}*y*l'_{2}$ and $w_{2}=l''_{1}* y*l_{2}*sx*l'_{1}$, where $l_{1}=l'_{1}l''_{1}$ and $l_{2}=l'_{2}l''_{2}$. This is due to the fact that  $l_{1}$ has no occurrence in $w_{2}$ and $l_{2}$ has no occurrence in $w_{1}$.  So, $w_{1}= l''_{2}* sx*l_{1}*y*l'_{2} \rightarrow l''_{2}* sx*m_{1}*y*l'_{2}\circlearrowleft^{i} sx*m_{1}*y*l'_{2}*l''_{2}\rightarrow sx*m_{1}*y*m_{2}$ and
$w_{2}=l''_{1}* y*l_{2}*sx*l'_{1} \rightarrow l''_{1}* y*m_{2}*sx*l'_{1} \circlearrowleft^{j} y*m_{2}*sx*l'_{1}*l''_{1}\rightarrow y*m_{2}*sx*m_{1}$.
We take then $z_{1}$ to be $sx*m_{1}*y*m_{2}$ and $z_{2}$ to be $y*m_{2}*sx*m_{1}$.\\
 Case $(iii)$: Assume that $\widetilde{w}= x*l_{1}*l''_{2}*y$, where $l''_{1}=l'_{2}$.
    There is an overlap ambiguity between these rules which resolve, since $\Re$ is complete:\\
    $\begin{array}{cccccccccc}
    &&l'_{1}l''_{1}l''_{2}\\
    &\swarrow &&\searrow\\
    m_{1}l''_{2} &&&& l'_{1}m_{2}\\
    &\searrow^{*}&&\swarrow^{*}\\
    &&z
    \end{array}$ \\
       The words $w_{1}$ and $w_{2}$ have the following form:
$w_{1}= l''_{2}* yx*l_{1}$
 and $w_{2}=l_{2}* yx*l'_{1}$.
 So, $w_{1}= l''_{2}* yx*l_{1} \rightarrow l''_{2}* yx*m_{1} \circlearrowleft^{i} m_{1}*l''_{2}* yx \rightarrow^{*} z*yx$
 and $w_{2}=l_{2}* yx*l'_{1}\rightarrow m_{2}* yx*l'_{1} \circlearrowleft^{j} l'_{1}*m_{2}* yx \rightarrow^{*} z*yx$.
So, we take $z_{1}$ and $z_{2}$ to be $z*yx$.\\
If we assume that both  $l_{1}$ and $l_{2}$ have an occurrence in $w_{1}$ but not in  $w_{2}$ (or the converse), then take  $\widetilde{w}$ to be  $w_{1}$ and the proof is done by a  case by case analysis of the same kind as above.
 \end{proof}
\begin{prop}\label{prop_triple_defined_converge}
Let $w$ be a word in $\Sigma^{*}$ and assume that $\operatorname{Allseq}$$(w)$ terminates.
Assume all the triples in
$\operatorname{Allseq}$$(w)$ are $\widetilde{c}$-defined, then $\operatorname{Allseq}$$(w)$ converges.
\end{prop}
\begin{proof}
We show in fact that the restriction of $\looparrowright$ to $\operatorname{Allseq}$$(w)$ is confluent.
From lemma \ref{lem_one_step}, the restriction of $\looparrowright$ to $\operatorname{Allseq}$$(w)$ is locally confluent since all the  triples in $\operatorname{Allseq}$$(w)$ are $\widetilde{c}$-defined.
From claim \ref{claim_local_conf},  $\looparrowright$  is confluent if and only if $\looparrowright$ is locally confluent, whenever it is terminating. So, using the same argument, we have that the restriction of $\looparrowright$ to $\operatorname{Allseq}$$(w)$ is confluent if and only if the restriction of $\looparrowright$ to $\operatorname{Allseq}$$(w)$ is locally confluent, whenever  $\operatorname{Allseq}$$(w)$ terminates.
So, the restriction of $\looparrowright$ to $\operatorname{Allseq}$$(w)$ is confluent, that is  $\operatorname{Allseq}$$(w)$ converges.
\end{proof}

\section{A necessary condition for the confluence of $\looparrowright$}
We find a necessary condition for the confluence of  $\looparrowright$, that is based on a  analysis of the rules in $\Re$. For that, we translate the signification of a triple that is not $\widetilde{c}$-defined in terms of the rules in $\Re$.
\begin{defn}
Let $w=x_{1}x_{2}x_{3}..x_{k}$ be a word, where the  $x_{i}$ are
generators for $1 \leq i \leq k$. Then we define the following
sets of words:\\
$\operatorname{pre}(w)=\{x_{1},x_{1}x_{2},x_{1}x_{2}x_{3},..,x_{1}x_{2}x_{3}..x_{k}\}$\\
$\operatorname{suf}(w)=\{x_{k},x_{k-1}x_{k},x_{k-2}x_{k-1}x_{k},..,x_{1}x_{2}x_{3}..x_{k}
\}$
 \end{defn}
  \begin{lem}\label{lem:no_Conj_pref_suff}
Let $(w,r_{1},r_{2})$ be a triple and let $l_{1}$ and $l_{2}$ denote the left-hand sides of the rules $r_{1}$ and $r_{2}$, respectively. If $\operatorname{pre}(l_{2})\cap \operatorname{suf}(l_{1}) = \emptyset$ or $\operatorname{pre}(l_{1})\cap
\operatorname{suf}(l_{2}) = \emptyset$, then the triple $(w,r_{1},r_{2})$ is $\widetilde{c}$-defined.
\end{lem}

\begin{proof}
From the assumption, $l_{1}$ is a subword of $w_{1}$ and $l_{2}$ is a subword of $w_{2}$, where $w_{1}$ and $w_{2}$ are cyclic conjugates of $w$.
We show that there exists a cyclic conjugate of $w$, $\widetilde{w}$, such that both $l_{1}$ and $l_{2}$ are subwords of $\widetilde{w}$.
If $\operatorname{pre}(l_{2})\cap \operatorname{suf}(l_{1}) = \emptyset$ or $\operatorname{pre}(l_{1})\cap
\operatorname{suf}(l_{2}) = \emptyset$, then there are three possibilities:\\
 $(i)$ If $\operatorname{pre}(l_{2})\cap \operatorname{suf}(l_{1}) = \emptyset$ and $\operatorname{pre}(l_{1})\cap
\operatorname{suf}(l_{2}) = \emptyset$, then there is no overlap ambiguity between the rules $r_{1}$ and $r_{2}$, so take
 $\widetilde{w}=xl_{1}yl_{2}z$, where $x,y,z$ are words, that makes $\widetilde{w}$ a cyclic conjugate of $w$.\\
 $(ii)$ If $\operatorname{pre}(l_{2})\cap \operatorname{suf}(l_{1}) \neq \emptyset$ and $\operatorname{pre}(l_{1})\cap
\operatorname{suf}(l_{2})= \emptyset$, then there is an overlap  ambiguity between the rules $r_{1}$ and $r_{2}$.
 We denote $l_{1}=l'_{1}l''_{1}$ and $l_{2}=l'_{2}l''_{2}$ and from the assumption $\operatorname{pre}(l'_{1})\cap \operatorname{suf}(l''_{2})= \emptyset$. Assume there is an overlap where $l''_{1}=l'_{2}$. So, take $\widetilde{w}=xl'_{1}l''_{1}l''_{2}y$, where $x,y$ are words, that makes $\widetilde{w}$ a cyclic conjugate of $w$.\\
   $(iii)$ If $\operatorname{pre}(l_{2})\cap \operatorname{suf}(l_{1}) = \emptyset$ and $\operatorname{pre}(l_{1})\cap
\operatorname{suf}(l_{2}) \neq \emptyset$, then this  is symmetric to case $(ii)$.
  \end{proof}
Note that if $\operatorname{pre}(l_{2})\cap \operatorname{suf}(l_{1}) \neq \emptyset$ and $\operatorname{pre}(l_{1})\cap
\operatorname{suf}(l_{2}) \neq \emptyset$, then it does not necessarily imply
that all the triples of the form $(w,r_{1},r_{2})$ are not $\widetilde{c}$-defined.
Yet, as the following example and lemma show, there exists a triple $(w,r_{1},r_{2})$ that is not $\widetilde{c}$-defined.
\begin{ex}\label{ex_wirt_tripledefined}
 Let $\Re= \{xy \rightarrow zx, yz \rightarrow zx, xz^{n}x  \rightarrow zxzy^{n-1}, n \geq 1\}$ from Ex. \ref{ex_def_triple}. The rules $xz^{2}x\rightarrow zxzy$ and $xz^{3}x\rightarrow zxzy^{2}$ satisfy
 $\operatorname{pre}(xz^{2}x)\cap \operatorname{suf}(xz^{3}x)=\{x\}$ and $\operatorname{pre}(xz^{3}x)\cap \operatorname{suf}(xz^{2}x)=\{x\}$. Yet, the triple $(xz^{2}xz^{3}x,\allowbreak xz^{2}x \rightarrow zxzy, xz^{3}x \rightarrow zxzy^{2})$ is $\widetilde{c}$-defined, but the triple  $(xz^{2}xz^{3},xz^{2}x \rightarrow zxzy, xz^{3}x \allowbreak \rightarrow zxzy^{2})$ is not $\widetilde{c}$-defined.
\end{ex}
\begin{lem}
Let $r_{1}$ and $r_{2}$ be rules in $\Re$ and we denote by $l_{1}$ and $l_{2}$ the left-hand sides of the rules $r_{1}$ and $r_{2}$ respectively.
Assume that $\operatorname{pre}(l_{1})\cap \operatorname{suf}(l_{2})\supseteq\{x\}$
and $\operatorname{pre}(l_{2})\cap \operatorname{suf}(l_{1})\supseteq\{x'\}$, where $x,x'$ are non-empty words.
Then there is a triple $(w,r_{1},r_{2})$ that is not $\widetilde{c}$-defined.
\end{lem}
\begin{proof}
We have that $l_{1}=xyx'$ and $l_{2}=x'vx$, where
$y,v$ are words and $x,x'$ are non-empty words.
Take $w$ to be  the word $xyx'v$, then it has no cyclic conjugate such that both the rules $r_{1}$ and $r_{2}$ can be applied on it.
\end{proof}
\begin{lem}\label{lem_notdefined_form_cyclicaloverlap}
Let $(w,r_{1},r_{2})$ be a triple and we denote by $l_{1}$ and $l_{2}$ the left-hand sides of the rules $r_{1}$ and $r_{2}$, respectively.
Assume that $(w,r_{1},r_{2})$ is not $\widetilde{c}$-defined. Then  $l_{1}=xuy$ and  $l_{2}=yvx$, where $u,v$ are words and $x,y$ are non-empty words.
\end{lem}
\begin{proof}
The triple $(w,r_{1},r_{2})$ is not $\widetilde{c}$-defined, so from lemma \ref{lem:no_Conj_pref_suff},
 $\operatorname{pre}(l_{2})\cap \operatorname{suf}(l_{1}) \neq \emptyset$ and $\operatorname{pre}(l_{1})\cap
\operatorname{suf}(l_{2}) \neq \emptyset$. Assume that  $\operatorname{pre}(l_{2})\cap \operatorname{suf}(l_{1})\supseteq \{x\}$ and $\operatorname{pre}(l_{1})\cap
\operatorname{suf}(l_{2}) \supseteq \{y\}$, where $x,y$ are non-empty words.
So, $l_{1}$ and $l_{2}$ have one of the following forms:\\
$(i)$ $l_{1}=xuy$ and  $l_{2}=yvx$, where $u,v$ are words.\\
$(ii)$ $l_{1}=xy$ and  $l_{2}=yx''$, where $x=x'x''$, $y=y'y''$ and $y''=x'$.\\
$(iii)$ $l_{1}=xy''$ and  $l_{2}=yx$, where $x=x'x''$, $y=y'y''$ and $x''=y'$.\\
$(iv)$ $l_{1}=xy''$ and  $l_{2}=yx''$, where $x=x'x''$, $y=y'y''$, and $y''=x'$, $x''=y'$.\\
We show that only case $(i)$ occurs, by showing that in the cases $(ii)$, $(iii)$ and $(iv)$  the triple $(w,r_{1},r_{2})$ is  $\widetilde{c}$-defined. This is done by describing $\widetilde{w}$ on which both $r_{1}$ and $r_{2}$ can be applied.
In any case, $w_{1}$ has to contain an occurrence of $l_{1}$ and $w_{2}$ has to contain an occurrence of $l_{2}$, where $w_{1}$ and $w_{2}$ are cyclic conjugates of $w$. In case $(ii)$, $l_{1}=x'x''y'y''$ and  $l_{2}=y'y''x''$, where  $y''=x'$, so there exists  $\widetilde{w}=x'x''y'y''x''$ such that it  contains one occurrence of $l_{1}$ and one occurrence of $l_{2}$. Case $(iii)$ is symmetric to case $(ii)$ and we consider  case $(iv)$.
In case $(iv)$,  $l_{1}=x'x''y''$ and  $l_{2}=y'y''x''$, where  $y''=x'$ and  $x''=y'$, so using the same argument as before, take $\widetilde{w}$ to be $x'x''y''x''$.
So, case $(i)$ occurs and $w$ has the form $xuyv$.
\end{proof}
\begin{defn}
We say that there is a \emph{cyclical overlap} between rules, if
there are two rules in $\Re$ of the form $xuy \rightarrow u'$ and
$yvx \rightarrow v'$, where $u',v'$ are words, $u,v,x,y$ are non-empty words and such that
$u'v$ and $v'u$ are not cyclic conjugates in $\Sigma^{*}$.
We say that there is a \emph{cyclical inclusion} if there are
two rules in $\Re$, $l\rightarrow v$ and $l' \rightarrow v'$,
where $l,v,l',v'$ are words and
 $l'$ is a cyclic conjugate of $l$ or $l'$ is a proper subword of a cyclic conjugate of
 $l$.  Whenever $l'$ is a cyclic conjugate of $l$, $v$ and $v'$ are not cyclic
 conjugates in $\Sigma^{*}$ and whenever $l'$ is a proper subword of $l_{1}$, where $l_{1}$ is a cyclic conjugate of
 $l$ (there is a non-empty word $u$ such that $l_{1}=ul'$),
then it holds that $l \rightarrow r$ and $l\circlearrowleft^{i}
l_{1}=ul' \rightarrow
uv'$ and $v$ and $uv'$ are not cyclic conjugates in $\Sigma^{*}$.
\end{defn}
\begin{ex}
In Example \ref{ex_wirt_tripledefined}, there is a cyclical overlap between the rules $xz^{2}x \rightarrow zxzy$ and $xz^{3}x \rightarrow zxzy^{2}$. In Example \ref{ex_not_cyc_confluent}, there is a cyclical inclusion between the rules $ab \rightarrow \underline{ab}$ and $ba \rightarrow \underline{ba}$.
\end{ex}
\begin{lem}\label{lem_pref_overlap}
Let $(w,r_{1},r_{2})$ be a triple  and let  $l_{1}$ and $l_{2}$ be the left-hand sides of the rules $r_{1}$ and $r_{2}$, respectively.
Assume that the triple $(w,r_{1},r_{2})$ is not  $\widetilde{c}$-defined.
Then there is a cyclical overlap or a cyclical inclusion between $r_{1}$ and $r_{2}$.
\end{lem}
\begin{proof}
The triple $(w,r_{1},r_{2})$ is not $\widetilde{c}$-defined, so from lemma \ref{lem_notdefined_form_cyclicaloverlap}, $l_{1}=xuy$ and $l_{2}=yvx$, where $x,y$ are non-empty words and
$u,v$ are words. If $u$ and $v$ are both the empty word, then
$l_{1}$ and $l_{2}$ are cyclic conjugates, that is there is a
cyclical inclusion. If $u$ is the empty word but $v$ is not the
empty word, then $l_{1}=xy$ and $l_{2}=yvx$, which means that
$l_{1}$ is a subword of a cyclic conjugate of $l_{2}$. So, in this
case and in the symmetric case (that is  $v$ is the empty word but $u$
is not the empty word) there is a cyclical inclusion. If none of
$u$ and $v$ is the empty word, then $l_{1}=xuy$ and $l_{2}=yvx$,
that is there is a cyclical overlap between these two rules.
\end{proof}
\begin{prop}\label{prop_nooverlap_allseq_converges}
Let $w$ be a word in $\Sigma^{*}$
 and assume that $\operatorname{Allseq}$$(w)$ terminates.
  If there are no cyclical overlaps and cyclical inclusions in $\operatorname{Allseq}$$(w)$, then $\operatorname{Allseq}$$(w)$ converges.
\end{prop}
\begin{proof}
If $\operatorname{Allseq}$$(w)$ does not converge, then from Proposition \ref{prop_triple_defined_converge}, this implies that
there is a triple $(w,r_{1},r_{2})$ in $\operatorname{Allseq}$$(w)$ that is not $\widetilde{w}-$defined.
 From lemma \ref{lem_pref_overlap}, this implies that there is a cyclical overlap or a cyclical inclusion in $\operatorname{Allseq}$$(w)$.
\end{proof}
Note that the converse is not necessarily true, that is there may be a cyclical overlap or a cyclical inclusion in $\operatorname{Allseq}$$(w)$ and yet a unique cyclically irreducible form is achieved in $\operatorname{Allseq}$$(w)$, as in the following example.
\begin{ex}
Let $\Re=\{ bab\rightarrow aba,ba^{n}ba\rightarrow aba^{2}b^{n-1}, n \geq 2\}$.
Let $w=ba^{2}ba$, then $\operatorname{Allseq}$$(w)$ does not terminate (see Ex. \ref{ex_braid_B3}). The triple $(w,bab\rightarrow aba, ba^{2}ba \rightarrow aba^{2}b )$ is not $\widetilde{c}-$defined since there is a cyclical inclusion of  the rule $bab\rightarrow aba$ in the rule $ba^{2}ba \rightarrow aba^{2}b$. Nevertheless, $w$ has a unique cyclically irreducible form $ba^{4}$ (up to $\bumpeq$): $ba^{2}ba\rightarrow aba^{2}b \circlearrowleft^{4} baba^{2}\rightarrow abaa^{2}$.
In fact, each $w=ba^{n}ba$ where $n \geq 2$ has a unique cyclically irreducible form $ba^{n+2}$ (up to $\bumpeq$).
\end{ex}

\begin{thm}\label{theo:conf_over_inc_amb_resolv}
Let $\Re$ be a complete and reduced linear rewriting system that is
cyclically terminating. If there are no rules in $\Re$ with cyclical overlaps or cyclical inclusions,
then $\Re$ is cyclically confluent.
\end{thm}
\begin{proof}
 From Proposition \ref{prop_nooverlap_allseq_converges}, if there are no rules in $\Re$ with cyclical overlaps or cyclical inclusions then $\operatorname{Allseq}$$(w)$ converges for all $w$.
 Since $\Re$ is cyclically terminating, $\Re$ is cyclically confluent if and only if
$\operatorname{Allseq}$$(w)$ converges for all $w$, so the proof is done.
\end{proof}

\section{The algorithm of cyclical completion}
 Knuth and Bendix have elaborated an
algorithm which for a given finite and terminating  rewriting system $\Re$, tests its completeness and if $\Re$ is not complete
then new rules are added to complete it. Instead of testing the
confluence of $\Re$, the algorithm tests the locally confluence of
$\Re$, since for a terminating  rewriting system  locally
confluence and confluence are equivalent. Two  rewriting systems
$\Re$ and $\Re'$ are said to be \emph{equivalent} if :\
$w_{1}\leftrightarrow^{*}w_{2}$ modulo $\Re$ if and only if
$w_{1}\leftrightarrow^{*}w_{2}$ modulo $\Re'$. So, by applying the
Knuth-Bendix algorithm on a terminating  rewriting system $\Re$ a
complete  rewriting system $\Re'$ that is equivalent to $\Re$ can
be found. Our aim in this section is to provide an algorithm of
cyclical completion which is much inspired by  the Knuth-Bendix
algorithm of completion. \\Let $\Re$ be a complete,
reduced and cyclically terminating  rewriting system, we assume that $\Re$ is finite.
 From Theorem \ref{theo:conf_over_inc_amb_resolv},
 if there are no cyclical overlaps or cyclical inclusions then $\Re$ is cyclically
confluent. Nevertheless, if there is a cyclical overlap or a cyclical inclusion, we define when it resolves in the following way. We say that the cyclical overlap between the rules $xuy \rightarrow u'$ and
$yvx \rightarrow v'$, where $u,v,u',v'$ are words, $x,y$ are non-empty words \emph{resolves} if there exist cyclically conjugate words $z$ and $z'$ such that $u'v\looparrowright^{*}z$ and
$uv'\looparrowright^{*}z'$.
If there is a  cyclical inclusion between the rules $l\rightarrow v$ and $l' \rightarrow v'$,
where $l,v,l',v'$ are words and
 $l'$ is a cyclic conjugate of $l$ or $l'$ is a proper subword of a cyclic conjugate of
 $l$, then we say that it resolves if there exist cyclically conjugate words $z$ and $z'$ such  that $v\looparrowright^{*}z$ and
$v'\looparrowright^{*}z'$ in the first case or
$v\looparrowright^{*}z$ and $uv'\looparrowright^{*}z'$
in the second case ($z \bumpeq z'$).
\begin{ex}
We consider the complete and finite rewriting system from Ex. \ref{ex_not_cyc_confluent}. Since there is a cyclical inclusion between the rules $ab \rightarrow \underline{ab}$ and $ba \rightarrow \underline{ba}$, it holds that  $ab \looparrowright \underline{ab}$
and $ab \looparrowright \underline{ba}$, where $\underline{ab}$ and $\underline{ba}$ are cyclically irreducible. We can decide arbitrarily wether  $\underline{ab} \looparrowright^{+}\underline{ba}$ or $\underline{ba} \looparrowright^{+}\underline{ab}$, in any case this cyclical inclusion resolves.
\end{ex}
In the following, we describe the algorithm of cyclical completion in which we add some new cyclical reductions. We denote by $\Re^{+}$ the  rewriting system with the added cyclical reductions and we add ``$+$'' in $\looparrowright^{+}$ for each cyclical reduction that is added in the process of cyclical completion. We assume that  $\Re$ is a finite, complete, reduced and cyclically terminating  rewriting system. The algorithm is described in the following.\\
$(i)$ If there are no cyclical overlaps or cyclical inclusions, then $\Re$ is cyclically
confluent, from Theorem \ref{theo:conf_over_inc_amb_resolv} and $\Re^{+}=\Re$.\\
$(ii)$ Assume there is a cyclical overlap or a cyclical inclusion in the word $w$:
$ w \looparrowright z_{1}$ and $ w \looparrowright z_{2}$.\\
 - if $z_{1}$ and $z_{2}$ are cyclically irreducible, then let decide $z_{1}\looparrowright^{+} z_{2}$ or $z_{2}\looparrowright^{+} z_{1}$.
     If at a former step, no $z_{i}\looparrowright^{+} u$ or $u\looparrowright^{+} z_{i}$ for $i=1,2$ was added, then we can decide arbitrarily wether $z_{1}\looparrowright^{+} z_{2}$ or $z_{2}\looparrowright^{+} z_{1}$.
     As an example, if $z_{1}\looparrowright^{+} u$ was added, then we choose $z_{2}\looparrowright^{+} z_{1}$.\\
    - if $z_{1}$ or $z_{2}$ is not cyclically irreducible, then
     we reduce them to their cyclically irreducible form $z'_{1}$ and $z'_{2}$
     and we decide wether $z'_{1}\looparrowright^{+} z'_{2}$ or $z'_{2}\looparrowright^{+} z'_{1}$ in the same way as before.\\
    The algorithm fails if the addition of a new cyclical reduction creates a contradiction:  assume  $z_{1}$ and $z_{2}$ are cyclically irreducible and we need to add $z_{1}\looparrowright^{+} z_{2}$ or $z_{2}\looparrowright^{+} z_{1}$ but   $z_{1}\looparrowright^{+} u$ and  $z_{2}\looparrowright^{+} v$ are already in $\Re^{+}$. In the Knuth-Bendix algorithm of completion, the addition of the new rules may create some additional overlap or inclusion ambiguities. We show in the following that this is not the case with the algorithm of cyclical completion and this is due to the fact that the relation $\looparrowright$ is not compatible with concatenation. From lemma \ref{lem_uloopvimplies conj}, if $u \looparrowright^{*} v$ then $u\equiv_{M} v$. In the following lemma, we show that this holds also with  $\looparrowright^{+}$.
\begin{lem}\label{lem_new_cycl_rule_equiv}
Let $\Re$ be a complete, reduced and cyclically terminating rewriting system. We
assume that $\Re$ is finite.
Let $\Re^{+}$ be the cyclical rewriting system obtained from the application of the algorithm of cyclical completion on $\Re$. If $u \looparrowright^{+} v$ then $u \equiv_{M} v$ modulo $\Re$.
\end{lem}
\begin{proof}
There are two cases to check: if $u \looparrowright^{+} v$ and if $u \looparrowright^{+}u_{2} \looparrowright^{+}u_{3}.. \looparrowright^{+} v$.
 If $u \looparrowright^{+} v$, then from the algorithm of cyclical completion, there is a word $w$ such that $w \looparrowright^{*} u$ and $w \looparrowright^{*} v$. So, $w \equiv_{M} u $ and $w \equiv_{M} v $ modulo $\Re$ from lemma \ref{lem_uloopvimplies_transp} and since $\equiv_{M}$ is an equivalence relation, $u \equiv_{M} v$ modulo $\Re$. If $u \looparrowright^{+}u_{2} \looparrowright^{+}u_{3}.. u_{k} \looparrowright^{+} v$, then from the first case we have $u \equiv_{M} u_{2}$, $u_{2} \equiv_{M} u_{3}$, ...,$u_{k} \equiv_{M} v$ modulo $\Re$. So, from the transitivity of $\equiv_{M}$, $u \equiv_{M} v$ modulo $\Re$.
\end{proof}

Given two complete, reduced and cyclically terminating rewriting systems $\Re$ and $\Re'$, we say that $\Re$ and $\Re'$ are \emph{cyclically equivalent} if the following condition holds: $u\equiv_{M} v$ modulo $\Re'$ if and only if $u\equiv_{M} v$ modulo $\Re$. We show that the cyclical rewriting system  $\Re^{+}$ obtained  from the application of the algorithm of cyclical completion on $\Re$ is cyclically equivalent to $\Re$.

\begin{lem}
Let $\Re$ be a complete, reduced and cyclically terminating  rewriting system, we assume that $\Re$ is finite.
Let $\Re^{+}$ be the cyclical  rewriting system obtained from the application of the algorithm of cyclical completion on $\Re$. Then $\Re^{+}$ and $\Re$ are cyclically equivalent, that is $u\equiv_{M} v$ modulo $\Re^{+}$ if and only if $u\equiv_{M} v$ modulo $\Re$.
\end{lem}
\begin{proof}
  It holds that $u\equiv_{M} v$ modulo $\Re$ if and only if there are words $x,y$ in $\Sigma^{*}$ such that $ux=_{M}xv$ and $yu=_{M}vy$.  Since the (linear) rules in $\Re^{+}$ are the same as those in $\Re$, this holds if and only if $u\equiv_{M} v$ modulo $\Re^{+}$ also.
\end{proof}
 We say that there is a \emph{cyclical ambiguity} in $w$ if
$w\looparrowright^{*}u$ and $w\looparrowright^{*}v$, where $u$ and $v$ are not cyclic conjugates.
If there exist cyclically conjugate words $z$ and $z'$ in $\Sigma^{*}$ such that
$u\looparrowright^{*}z$ and $v\looparrowright^{*}z'$,  then we say that this cyclical ambiguity \emph{resolves}. Clearly,  a rewriting system is cyclically confluent if and only if all the cyclical ambiguities resolve. Now, we show that whenever the algorithm of cyclical completion does not fail, the rewriting system obtained $\Re^{+}$ is cyclically complete.
\begin{prop}\label{thm_proof_algo_compl}
Let $\Re$ be a complete, reduced and cyclically terminating  rewriting system, we
assume that $\Re$ is finite.
Let $\Re^{+}$ be the cyclical rewriting system obtained from the application of the algorithm of cyclical completion on $\Re$. Then $\Re^{+}$ is cyclically complete.
\end{prop}
\begin{proof}
 We need to show that $\Re^{+}$ is cyclically confluent.
Clearly, by the application of the algorithm of cyclical completion on $\Re$ the cyclical overlaps and inclusions in $\Re$ are resolved. So, it remains to show that the addition of the new cyclical rules in $\Re^{+}$ does not create a cyclical ambiguity.  If a cyclical ambiguity occurs, then there should be one of the following kind of rules in $\Re^{+}$:\\
- $u \looparrowright^{+} v$ and $l\rightarrow x$, where $l\bumpeq u$.\\
- $u \looparrowright^{+} v$ and $l \looparrowright^{+} x$, where $l\bumpeq u$.\\
The first case cannot occur, since $u$ is cyclically irreducible modulo $\Re$ and the second case cannot occur, since in this case the algorithm of cyclical completion fails.
\end{proof}
\section{Length-preserving rewriting systems}
We say that a rewriting system $\Re$ is
\emph{length-preserving} if $\Re$ satisfies the condition that
the left-hand sides of rules have the same length as their
corresponding right-hand sides. We show that if $\Re$ is a length-preserving
 rewriting system, then an infinite sequence of cyclical reductions occur only if there is a
 repetition of some word in the sequence or if a word and its cyclic conjugate occur there. Using this fact, we define an equivalence relation on the words that permits us to obtain some partial results in the case that  $\Re$ is not cyclically terminating.
\begin{lem}\label{lem_length_termin}
Let $\Re$ be a complete,
reduced  rewriting system that is length-preserving. Let
 $u\looparrowright u_{1}\looparrowright u_{2}...\looparrowright u_{n}$ be a sequence such that $u_{i}$ and $u_{j}$ are not cyclic conjugates in
$\Sigma^{*}$ for $i\neq j$.
Then this sequence terminates.
\end{lem}
\begin{proof}
 From the assumption,
applying $\Re$ on $u$ does not change its length $\ell(u)$, so we have that
\emph{$\ell(u_{i})= \ell(u)$} for all $i$. Since the number of
words of length \emph{$\ell(u)$} is finite and there is no occurrence of words and their cyclic conjugates  in the sequence, the sequence terminates.
\end{proof}
Note that using the same argument as in lemma \ref{lem_length_termin}, we have that if  $\Re$ is length-decreasing, that is  all the left-hand sides of rules have  length greater than their
corresponding right-hand sides, then there is no infinite sequence of cyclical reductions, that is $\Re$ is cyclically terminating. In the following lemma, we show  that if there is an infinite sequence of cyclical reductions that results from the occurrence of a word $w$ and its cyclic conjugate $\widetilde{w}$, then there are some relations of commutativity involving $w$ and $\widetilde{w}$.
This is not clear if these relations of commutativity are a sufficient condition for the occurrence of an infinite sequence, nor if such a sufficient condition can be found.
\begin{lem}
Assume there is an infinite sequence $w \looparrowright^{*}\widetilde{w}$,
where $w \bumpeq \widetilde{w}$.
Then there are words $x,y$ such that $yx\widetilde{w}=_{M}\widetilde{w}yx$ and
 $xyw=_{M}wxy$.
\end{lem}
\begin{proof}
From lemma \ref{lem_uloopvimplies conj},  $w \equiv_{M} \widetilde{w}$, that is  there are words $x,y$ in $\Sigma^{*}$ such that $wx=_{M}x\widetilde{w}$ and $yw=_{M}\widetilde{w}y$. So, $wxy=_{M}x\widetilde{w}y=_{M}xyw$ and $yx\widetilde{w}=_{M}ywx=_{M}\widetilde{w}yx$.
\end{proof}

We now define the following equivalence relation $\sim$ on $\Sigma^{*}$. Let $u,v$ be different words in $\Sigma^{*}$.
 We define $u \sim v$ if and only if $u\looparrowright^{*}v$ and
 $v\looparrowright^{*}u$, where at least one rule from $\Re$ has been applied in these sequences.
 Clearly, the relation $\sim$ is symmetric and transitive. In order to make it reflexive,
 we define $u\sim u$, that is $u\looparrowright^{*}u$ in an empty
 way and also $u\sim \widetilde{u}$, where $u\bumpeq \widetilde{u}$.
 Clearly, if $\Re$ is cyclically terminating, then each
 equivalence class contains a single word, up to $\bumpeq$.
Now, we show that there is  a  partial solution to the left and right conjugacy problem, using the equivalence relation $\sim$ in the case that $\Re$ is not cyclically terminating. Note that given a word $w$ such  that $\operatorname{Allseq}$$(w)$ does not terminate, it may occur one of the following; either there is no cyclically irreducible form achieved in $\operatorname{Allseq}$$(w)$ (as in Ex. \ref{ex_not_termin_no_cyc_irred}) or there is a unique cyclically irreducible form achieved in $\operatorname{Allseq}$$(w)$ (as in  Ex. \ref{ex_braid_uniquecyclicalform}).
 \begin{thm}
 Let  $u$ and $v$ be in $\Sigma^{*}$. If there exists a word $z$ such that $u \sim z$ and $v \sim z$ then $u \equiv_{M} v$.
 \end{thm}
 \begin{proof}
 If there exists a word $z$ such that $u \sim z$ and $v \sim z$, then from the definition of $\sim$ there are sequences $u \looparrowright^{*} z $ and $v \looparrowright^{*}z$. So, $u \equiv_{M} z$ and $v \equiv_{M} z$, from lemma \ref{lem_uloopvimplies conj} and since $\equiv_{M}$ is an equivalence relation $u \equiv_{M} v$.
 \end{proof}
Note that the converse is not true as the following example illustrates it.
\begin{ex}
Let  $\Re=\{ bab\rightarrow aba,ba^{n}ba\rightarrow aba^{2}b^{n-1}, n \geq 2\}$.
It holds that $a\equiv_{M}b$, since $a(aba)=_{M}(aba)b$ and
$(aba)a=_{M}b(aba)$. Yet, there is no sequence of cyclical reductions such that $a \sim b$. This example is due to Patrick Dehornoy.
\end{ex}
We can consider a  rewriting system that is not length increasing (that is  all the rules preserve or decrease the length) to be \emph{cyclically terminating up to $\sim$} and apply on it the algorithm of cyclical completion and obtain that it is \emph{cyclically complete up to $\sim$}. This is due to the fact that  also in this case infinite cyclical sequences would result from the occurrence of a word and its cyclic conjugate.  If there exists a cyclically irreducible form then it is unique, but the existence of a cyclically irreducible form is not ensured. Indeed, let consider the complete and finite rewriting system $\Re$ from Ex. \ref{ex_not_cyc_confluent}. It holds that  $\Re$ is not length increasing and not cyclically terminating, since there are infinite sequences of cyclical reductions (as an example $\Delta a \rightarrow b \Delta \circlearrowleft ^{1} \Delta b \rightarrow a \Delta$). By applying the algorithm of cyclical completion on $\Re$, we have $\Re^{+}=\Re \cup \{\underline{ab} \looparrowright^{+} \underline{ba}\}$ is cyclically complete up to $\sim$, but yet   $\Delta a$ has no cyclically irreducible form.

\bibliographystyle{amsplain}
\bibliography{bibliographie_15Juillet2009}

\end{document}